\documentclass[12pt]{amsart}
\usepackage{amsmath, amssymb}
\usepackage{graphicx}
\usepackage{url}

\newtheorem{theorem}{Theorem}
\newtheorem{conjecture}[theorem]{Conjecture}
\newtheorem{observation}[theorem]{Observation}
\newtheorem{corollary}[theorem]{Corollary}

\newtheorem{lemma}[theorem]{Lemma}

\newtheorem{open}[theorem]{Open Problem}
\newtheorem{definition}{Definition}

\theoremstyle{definition}

\def\bth{\begin{theorem}}
\def\eth{\end{theorem}}
\def\bc{\begin{corollary}}
\def\ec{\end{corollary}}
\def\bcj{\begin{conjecture}}
\def\ecj{\end{conjecture}}

\title[]{Degree Game for Special Regular Graphs}

\author[Gy\H{o}rffy]{Lajos Gy\H{o}rffy }\address{Bolyai Institute,
University of Szeged and John von Neumann University, Kecskemét}

\email{lgyorffy@math.u-szeged.hu}


\begin{document}

\begin{abstract}

For a given $d$-regular graph $G$, a Maker-Breaker degree game is played by two players who alternately claim previously unclaimed edges of $G$. In the standard variant, the goal of Maker is to maximize the maximum degree of their induced subgraph, while Breaker aims to minimize it, or equivalently, to guarantee a certain minimum degree in their own subgraph. A classic pairing strategy shows that Breaker can secure at least $\lfloor d/4 \rfloor$ edges at every vertex of any $d$-regular graph. Breaking this bound for general or even for specific classes of graphs has been a long-standing open problem in combinatorial game theory; indeed, J. Beck characterized this challenge in his monograph as the first among the seven most humiliating open problems of positional game theory.
In this paper, we improve the $d/4$ bound for some infinite graph families, such as the hypercube graph $Q_d$, grids and tori. We first show that Breaker can secure a degree of one at every vertex in $Q_3$, then lift this to higher dimensions, where Breaker can guarantee a degree of at least $\lfloor d/3 \rfloor$.
\end{abstract}

\maketitle

\section{Introduction}

In the field of combinatorial game theory, positional games on graphs have attracted significant attention over the past few decades, heavily pioneered by Beck \cite{BB} and Krivelevich et al. \cite{Hefetz}. Among these, Maker-Breaker games represent a central paradigm. In an unbiased $(1:1)$ Maker-Breaker game played on the edge set of a base graph $G$, two players, Maker and Breaker, take turns claiming one free edge of $G$ per turn.

A natural and structurally fundamental variant is the \emph{degree game}. In this setting, the players are interested in the local degree distribution of their subgraphs at the end of the game. Specifically, Breaker's objective is to prevent Maker from creating a vertex with an excessively high degree. From Breaker's perspective, this is equivalent to securing a guaranteed minimum degree in Breaker's own induced subgraph at every single vertex of $G$.

When the game is played on dense graphs, such as the complete graph $K_n$, global potential function arguments allow Breaker to achieve a near-perfect split, restricting Maker's maximum degree to $n/2 + o(n)$. For decades, the best known general bound for Breaker on an arbitrary $d$-regular graph has been the classic $\lfloor d/4 \rfloor$ barrier. This bound is achieved via a simple, static pairing strategy. One fixes an Euler orientation of $G$ (where each vertex has out-degree $\lceil d/2 \rceil$) and partitions the outgoing edges into $\lfloor d/4 \rfloor$ disjoint pairs. Whenever Maker claims one edge of a pair, Breaker immediately responds by claiming the other. 

While this pairing strategy is elegant, it is inherently local and rigid. A major open question in the literature has been whether Breaker can strictly beat this $d/4$ threshold—even by an $\epsilon \cdot d$ term—on any $d$-regular graphs, or at least on prominent, symmetric sparse graph families. 

In this paper, we break the $d/4$ bound by focusing on the $d$-dimensional hypercube graph $Q_d$, which lies on the ``border'' between sparse and dense structures, where the potential functions begin to fail. Our main result is the following:

\begin{theorem} \label{maint}
	In the $(1:1)$ Maker-Breaker degree game played on the hypercube $Q_d$, Breaker has a strategy to secure a degree of at least $\lfloor d/3 \rfloor$ at every vertex.
\end{theorem}

To achieve this, we introduce a hierarchical strategy that avoids the pitfalls of complex global potential functions on sparse graphs. First, we solve the case of $Q_3$ to show that Breaker can always guarantee a degree of at least one at every vertex, this establishes the base case. We then exploit the fact that the hypercube can be naturally decomposed via the Cartesian product operation ($Q_d \cong Q_3 \square Q_{d-3}$). By decomposing the game space into orthogonal copies of low-dimensional hypercubes, Breaker can play independent, parallel defense games. This product-based induction allows the local advantage gained in $Q_3$ to scale linearly with the dimension, effectively shifting the achievable bound from $d/4$ to $d/3$.

The rest of this paper is organized as follows. In Section 2, we contextualize the problem, review the preliminaries, and formally define the game and the notation used. In Section 3, we present the results. Section 4 details the explicit case analysis for the base case on $Q_3$, the Cartesian product decomposition and the proof of the main theorem for general dimensions. Finally, we conclude with some open remarks and future directions in Section 5.

\section{Preliminaries} \label{prel}

\subsection{Origins}
The study of degree games started with Paul Erd\H{o}s who raised the question as follows. Two players, {\em Maker} and {\em Breaker} taking turns, claim the edges of $K_n$ (or $K_{n, n}$), the complete (bipartite) graph on $n$ (or $2n$) vertices. Maker's goal is to maximize the maximum degree in Maker's own induced subgraph, in other words to get as many edges of their own at some vertices as possible. One writes this maximum as $n/2 +k$, since Maker obviously can have at least $n/2$, and the question is for which values can win Maker (Breaker)? 

The problem was investigated by Beck \cite{Beck2} and Sz\'ekely \cite{Szekely}, and they showed that Breaker wins if $k \geq \sqrt{n\log n}$ while Maker wins if $k < \sqrt{n}/15$.
For {\em biased} versions, see Balogh et al. see \cite{BMP}. 

It is natural to consider general graphs as well. In the following, we focus primarly on $d$-regular graphs where $d \in \mathbb{N}$. Let us formally define the Maker-Breaker degree game on a simple graph $G$.

\begin{definition}[Degree game]
	Let $G$ be an arbitrary simple graph. In the Maker-Breaker degree game, Maker and Breaker alternately claim previously unclaimed edges of $G$, one at a time. Maker's objective is to maximize the maximum degree of their induced subgraph (i.e., to claim as many edges incident to some vertex as possible), while Breaker's objective is to minimize it.
\end{definition}

In some cases, it may be more useful to consider an alternative formulation of the game where Breaker's goal is explicitly quantified. Specifically, Breaker aims to secure at least $t \in \mathbb{N}$ edges at every vertex by the end of the game, which is of particular interest when $t=1$ or $t=2$.

\begin{definition}[$t$-Degree Game]
	
	Let $G$ be an arbitrary simple graph. In the $t$-degree game, Maker and Breaker alternately claim previously unclaimed edges of $G$. Breaker's objective is to secure at least $t$ edges at every vertex. Conversely, Maker's objective is to claim at least $\deg_G(x) - t + 1$ edges at some vertex $x \in V(G)$, where $\deg_G(x)$ denotes the degree of $x$ in $G$.
	
	In the special case of the $1$-degree game, Maker aims to claim all edges incident to some vertex $x$ (i.e., $\deg_G(x)$ edges), while Breaker aims to secure at least one edge at every vertex.
		
\end{definition}  

We note that several closely related positional games have recently been explored on sparse and regular graphs. For instance, in the total domination game investigated by Forcan and Mikala\v{c}ki~ \cite{Forc2}, players alternately claim vertices instead of edges to optimize dominating configurations on cubic graphs. Conversely, in the ``Toucher and Isolator'' game introduced by Dowden et al. \cite{TI}, the mechanics strictly involve edge selection, where the primary objective is to evaluate the exact number of vertices that can be successfully isolated under adversarial play.  In fact, Toucher acts as our Breaker and Isolator plays as our Maker, where preventing any vertex isolation corresponds to a Breaker victory. The only algorithmic difference is the move order, as Toucher (Breaker) opens the game in their framework.

Another related approach appears in the papers of Hefetz, Krivelevich, Stojakovi\'c and Szab\'o~\cite{Hefetz}; and Balogh and Pluh\'ar~\cite{BP}, where the minimal graphs on which Breaker (Maker with their notation) can secure a positive degree are characterized.

To simplify notation, for any simple graph $G$, we define $\delta_B(G)$ to be the maximum degree that Breaker can secure at \emph{every} vertex. In other words, $\delta_B(G)$ denotes the largest minimum degree Breaker can guarantee in their own induced subgraph at the end of the game.

\begin{definition}
	For any simple graph $G$, we define $\delta_B(G)$ to be the maximum minimum degree that Breaker can guarantee in their own subgraph at the end of the game.
\end{definition}

\begin{observation}[Folklore, Beck \cite{BB}] \label{d4}
	If $G$ is a $d$-regular graph, then $\delta_B(G) \ge \lfloor d/4 \rfloor$.
\end{observation}

	

To provide further context, let us momentarily depart from Maker-Breaker games and consider a related edge-coloring problem. This problem highlights a significant divergence between game-theoretic outcomes and static graph colorings.

\begin{observation}[Folklore] \label{color}
	The edges of any $d$-regular graph $G$ can be $2$-colored such that at every vertex $v \in V(G)$, the number of incident red and blue edges is $\lfloor d/2 \rfloor \pm 1$.
\end{observation}

\begin{proof}
	First, suppose $d$ is even (and $G$ is connected). By constructing an Eulerian tour and coloring its edges in an alternating fashion, every vertex receives a balanced number of red and blue edges each time the tour transitions through it. Consequently, at every vertex, the number of incident red and blue edges is exactly $d/2$, except possibly for the starting vertex of the tour (which is also the ending vertex). For this singular vertex, the parity of the alternating colors may cause an imbalance, resulting in $d/2 \pm 1$ red and blue edges when $d \equiv 2 \pmod 4$.
	
	If $d$ is odd, we can add an auxiliary vertex connected to all vertices to obtain an even-regular graph, which similarly yields a distribution of $\lfloor d/2 \rfloor \pm 1$ colored edges at every original vertex.
\end{proof}

Returning to Maker-Breaker games, J. Beck stated in Chapter 49 of his monograph \cite{BB} (p. 655) that the first problem among \emph{the 7 most humiliating open problems} in positional game theory is the following (see also Open Problem 16.1 in Chapter 16): 
When considering the Maker–Breaker degree game played on an arbitrary $d$-regular graph, the challenge is to ``beat'' the pairing strategy and secure a better result for Breaker by establishing that $\delta_B(G) \ge \lfloor (1/4+\epsilon)d \rfloor$ for some $\epsilon > 0$.

\begin{open}[Beck \cite{BB} 16.1]
	Can we improve the lower bound of $\lfloor d/4 \rfloor$ in the degree game to some $c \cdot d$, where $c > 1/4$? Is it possible to secure a bound of $d/2 - o(d)$?
\end{open}

\subsection{Derandomization}

To fully appreciate the significance of degree games, another brief digression might be instructive. The primary motivation stems from the fact that degree games serve as highly effective auxiliary games when analyzing other complex positional games (see, e.g., \cite{BP, Cser, GSZ}). An even deeper connection lies in the interplay between the probabilistic method and game theory. Broadly speaking, when optimal strategic players are replaced by random ones, the expected outcome of the game remains remarkably close to the strategic equilibrium. This fundamental observation is a cornerstone of positional game theory, conceptualized as the \emph{probabilistic intuition} or the \emph{Erd\H{o}s paradigm} \cite{Beck2, erdos, KSZ}. 

In this probabilistic framework, randomness is typically generated via independent coin tossing; for instance, the edges of a graph (or more generally, the vertices of a hypergraph) are colored uniformly at random. From Breaker's perspective, a favorable coloring is one in which the colors of the edges incident to any given vertex are well-balanced. While finding such an appropriate static coloring can be viewed as a \emph{single-player} optimization problem, the subsequent challenge is to \emph{derandomize} this process—that is, to construct an explicit, deterministic winning strategy for one of the players in a competitive, two-player environment.

Adapting this point of view, it is straightforward to verify that the bounds $k \ge \sqrt{n\log n}$ and $k < \sqrt{n}/15$ mentioned at the beginning of Section~\ref{prel} correspond to the applications of the \emph{first moment} and \emph{second moment} methods, respectively. In contrast, the baseline $\delta_B(G) \ge \lfloor d/4 \rfloor$ arises from a perfect matching argument, representing a global structural phenomenon.

The derandomization of the Lov\'asz Local Lemma proved to be exceptionally challenging. Following the preliminary breakthroughs by Beck \cite{BeckLLL}, Moser and Tardos successfully achieved this derandomization for a large class of problems in single-player settings (such as graph colorings) in their seminal work \cite{MT}.

For games, Beck \cite{BB} formulated the so-called \emph{Neighborhood conjecture}. Although this conjecture was subsequently refuted in its general form by Gebauer \cite{Geb}, several weaker variants remain wide open (cf. \cite{BB}).

Perhaps the most accessible open variants are the degree games, which remain unsolved even when restricted to the class of $d$-regular graphs. Consequently, as noted by J\'ozsef Beck \cite{BB}, they occupy a prominent place among the most challenging open problems in the field.

If the edges of a graph $G$ (not necessarily $d$-regular) are colored uniformly at random via independent coin tossing, a standard application of the Lov\'asz Local Lemma guarantees that there exists a $2$-coloring in which the maximum color discrepancy at any vertex is less than $c\sqrt{\Delta \log \Delta}$, where $\Delta$ denotes the maximum degree of $G$. Hence, it is plausible to conjecture the following:

\begin{conjecture} (Folklore) \label{nirvana}
For any $d$-regular graph $G$, $\delta_B(G) \ge d/2 - c\sqrt{d\log d}$ holds for some positive constant $c$.
\end{conjecture}

\subsection{The d-dimensional hypercube}
Arguably, the most intriguing question in the field concerns the behavior of $\delta_B(G)$ when $G$ is a $d$-regular graph. A natural first step towards settling Conjecture~\ref{nirvana} is to gain a deeper understanding of the hypercube graph $Q_d$, which is formed by the vertices and edges of the $d$-dimensional hypercube. With its $2^d$ vertices and $d \cdot 2^{d-1}$ edges, $Q_d$ lies precisely on the ``border'' between sparse and dense graphs.

The aforementioned bound $\delta_B(G) \ge \lfloor d/4 \rfloor$ yields a lower bound for $Q_d$ as well. On the other hand, Beck's weight-function method yields the bound $\delta_B(G) \ge d/2 - \sqrt{d\log_2 |V(G)|}$ for any $d$-regular graph $G$. However, this latter bound becomes vacuous for the hypercube graph $Q_d$, yielding only the trivial estimate $\delta_B(Q_d) \ge 0$.

In fact, no one managed to get even $\delta_B(G) \geq \lfloor (1/4 +\epsilon)d \rfloor$  for an $\epsilon > 0$ in general in the last forty years. This is very frustrating since there are such two-colorings of the edges in any $d$-regular graph $G$, in which the difference of the numbers of colors is at most two at all vertices, as we mentioned above.

A related structural game on the hypercube was investigated by Naimi and Sundberg \cite{NS}, who analyzed pairing strategies where the winning sets are subcubes of $Q_d$. While their work focuses on blocking global subcube achievements, our framework addresses the complementary challenge of defending local vertex degrees under dynamic play.

If the graph were even sparser---for instance, the infinite $d$-regular tree $\mathcal{T}_d$---one would obtain $\delta_B(\mathcal{T}_d) = \lfloor (d-1)/2 \rfloor$ (according to Balogh \cite{Btree}) via a pairing strategy applied to the edges leading to the next level of a rooted tree.

In this paper, we improve the classical $\lfloor d/4 \rfloor$ baseline to $\lfloor d/3 \rfloor$ for a prominent class of well-known $d$-regular graphs, including the hypercube graph $Q_d$.

\section{Results}

For the hypercube graph $Q_d$ and other prominent infinite families of $d$-regular graphs, we establish explicit lower bounds that strictly outperform the classical $\lfloor d/4 \rfloor$ baseline.

\begin{figure}[htbp]
	\centering
	\includegraphics[scale=0.48]{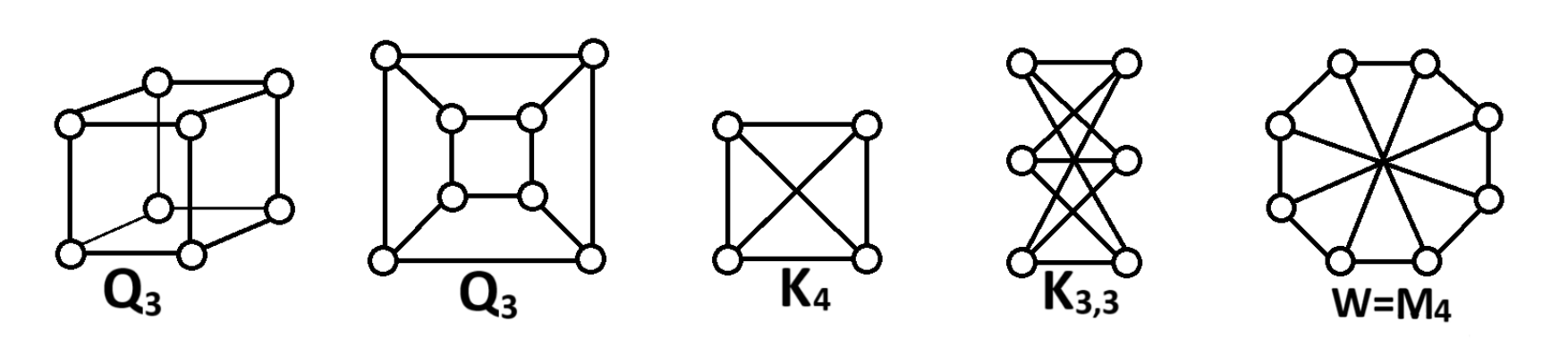}
	\caption {Some Breaker winner cubic graphs: $Q_3$ (twice), $K_4$, $K_{3,3}$, $W=M_4$}
	\label{Bwin}
\end{figure}

\begin{lemma} \label{cube} 
	Breaker has a winning strategy in the $1$-degree game played on the $3$-dimensional hypercube graph $Q_3$; that is, $\delta_B(Q_3) = 1$.
\end{lemma}

Our main result establishes that the $\lfloor d/4 \rfloor$ bound can be strictly improved to $\lfloor d/3 \rfloor$ for the entire family of hypercube graphs. We reformulate Theorem~\ref{maint}:

\begin{theorem} \label{d-cube}
	For any $d$-dimensional hypercube graph $Q_d$, we have $\delta_B(Q_d) \ge \lfloor d/3 \rfloor$.
	
\end{theorem}

One can also consider infinite grid graphs or toroidal grid graphs. Let $\mathbb{Z}^d$ denote the infinite $d$-dimensional grid graph, and let $\mathbb{T}^d_n$ denote the $d$-dimensional toroidal grid graph with $n^d$ vertices, obtained by identifying the opposite sides of a finite grid.

\begin{theorem} \label{d-grid}
	For the infinite $d$-dimensional grid graph $\mathbb{Z}^d$ and the $d$-dimensional toroidal grid graph $\mathbb{T}^d_n$ with $n \ge 2$ (both graphs being $2d$-regular), we have
	\[ \delta_B(\mathbb{Z}^d) \ge \left\lfloor \frac{2d}{3} \right\rfloor \quad \text{and} \quad \delta_B(\mathbb{T}^d_n) \ge \left\lfloor \frac{2d}{3} \right\rfloor. \]
\end{theorem}

There exist some other small-order $3$-regular graphs on which Breaker also has a winning strategy in the $1$-degree game.

\begin{observation} \label{small} 
	We have $\delta_B(G) \ge 1$ if $G$ is one of the following $3$-regular graphs: the complete graph $K_4$, the Wagner graph $W$ (also known as the $4$-Möbius ladder, $M_4$), or the complete bipartite graph $K_{3,3}$, as Figure \ref{Bwin} illustrates.
\end{observation}

Breaker's winning strategies for these graphs can be straightforwardly verified by the reader; hence, to keep the exposition concise, we omit the explicit case-by-case analysis. While the games on $K_4$ and $K_{3,3}$ are trivial, the analysis for $W$ involves 14 subcases with 26 subsubcases. Instead, we present a full, detailed case analysis for the hypercube graph $Q_3$ in Section 3.

\begin{figure}[htbp]
	\centering
	\includegraphics[scale=0.45]{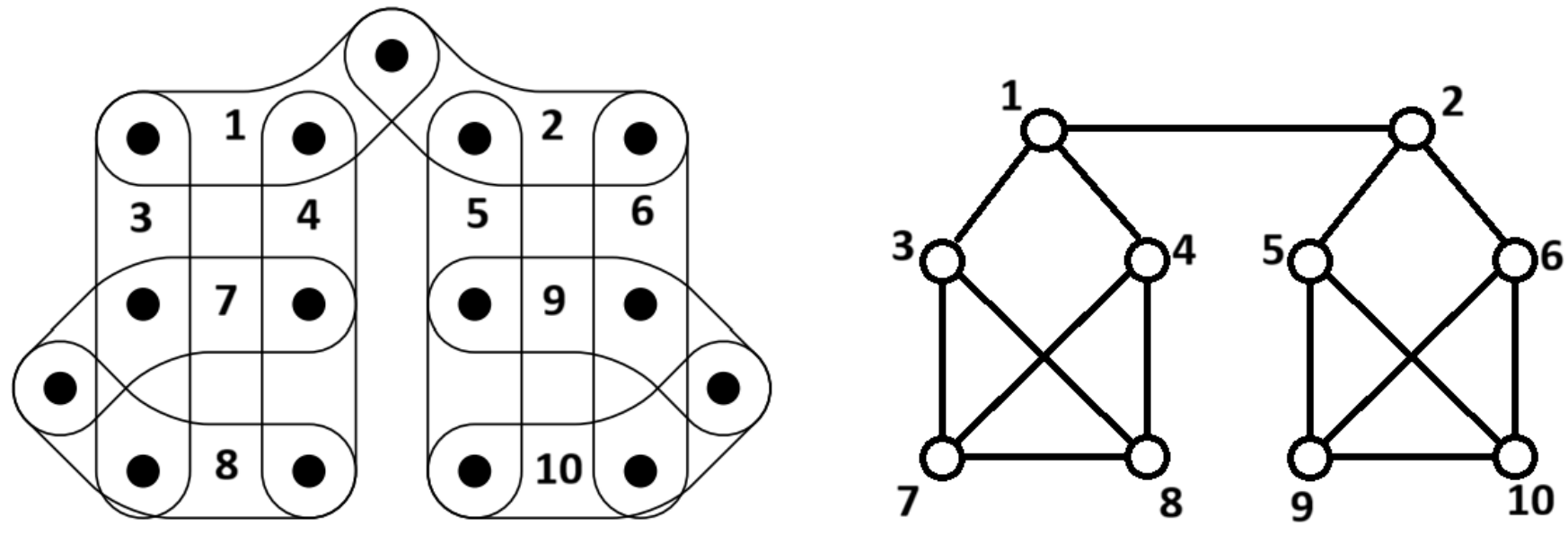}
	\caption {A Maker winner hypergraph of Knox \cite{Knox} and its graph version.}
	\label{kno}
\end{figure}

There exist several $3$-regular graphs on which Breaker cannot win the $1$-degree game as well; that is, at the end of the game, Breaker's induced subgraph will contain an isolated vertex. A classic example of such a structure can be derived from a Maker winner 3-uniform hypergraph of Knox \cite{Knox}. 

Note that the $1$-degree game played on a graph $G$ can be naturally formulated as a positional game on a hypergraph $H$, where the winning sets (hyperedges) of $H$ correspond to the vertices of $G$, and the ground set (vertices) of $H$ consists of the edges of $G$. Under this framework, Maker wins the $1$-degree game on $G$ (by claiming all edges incident to some vertex of $G$) if and only if Maker wins the corresponding hypergraph game on $H$ by fully claiming all vertices of a hyperedge. Conversely, Breaker wins the $1$-degree game on $G$ (by securing at least one edge at every vertex of $G$) if and only if Breaker wins the hypergraph game on $H$ by claiming at least one vertex in every hyperedge of $H$. 
To clarify this dual structure, Figure~\ref{kno} illustrates both the original $3$-uniform hypergraph introduced by Knox \cite{Knox} for the hypergraph game and the corresponding $3$-regular graph derived from it for the degree game.

\begin{observation} \label{hexa} 
	Maker can claim all three edges incident to some vertex in the $1$-degree game played on either the hexagonal lattice graph, the $5$-Möbius ladder $M_5$, or the graph derived from Knox's counterexample hypergraph \cite{Knox}. Furthermore, we conjecture that the same holds for the Petersen graph.
\end{observation}

On Knox's graph, Maker wins by first choosing edge $(1,2)$; then, if Breaker responds for example on the right side, Maker wins the game in three moves by choosing edge $(7,8)$ regardless of what Breaker plays.
Similarly, on the hexagonal lattice graph and on the $5$-Möbius ladder can be easily recovered through a manual inspection, which we omit here, because these do not play an important role in our paper.

We note that in the ``Toucher and Isolator'' game \cite{TI}, another explicit example of a Maker-winning cubic graph is identified, where Isolator (Maker) can force a vertex isolation on a specific $24$-vertex $3$-regular graph. Notably, this victory is achieved even under a framework where Toucher (Breaker) opens the game, making Maker's objective significantly more challenging.

This concludes the overview of our main theorems, observations, and examples. In the following section, we provide the explicit deterministic strategies and formal proofs for each of these theorems.

\section{Proofs}

\begin{figure}[htbp]
	\centering
	\includegraphics[scale=0.67]{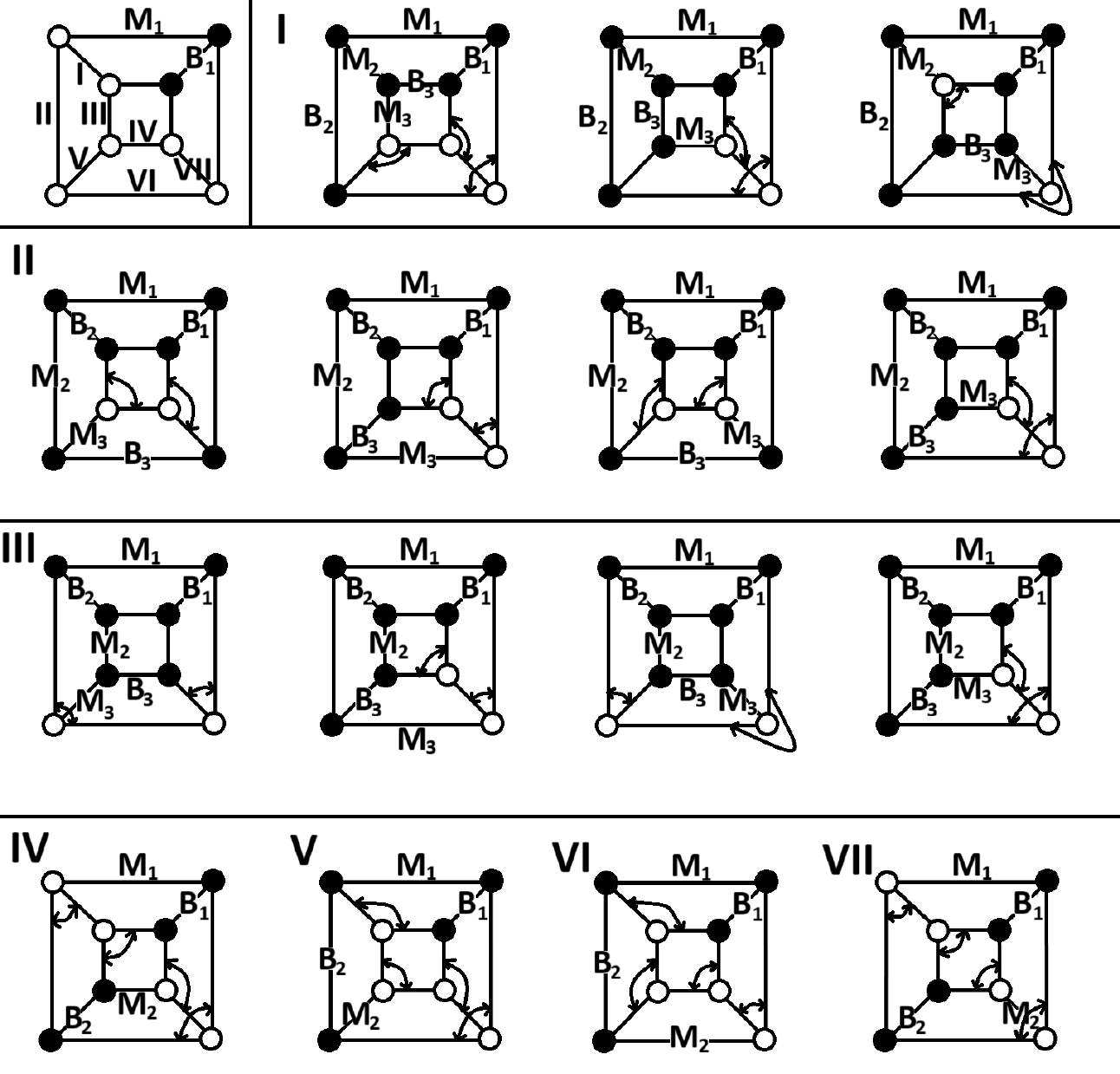}
	\caption {The seven cases}
	\label{Q3pr}
\end{figure}

\begin{proof}[Proof of Lemma~\ref{cube}, i.e.\ $\delta_B(Q_3) = 1$:]
In the context of the $1$-degree game, we define a vertex of $G$ to be \emph{dead} if Breaker has already claimed at least one edge incident to it. 
Conversely, a vertex is called \emph{alive} if all of its incident edges that have been claimed so far belong exclusively to Maker.

Claiming an edge between two dead vertices is completely suboptimal for both players. Furthermore, claiming an edge incident to two alive vertices is always strictly better than claiming an edge incident to exactly one of the previous two alive vertices and one other dead vertex, as formulated in the following straightforward lemma:

\begin{lemma}\label{jobb}
	Suppose that Maker (or Breaker) has a winning strategy by claiming an edge $x$ that has one dead and one alive endpoint. If there exists another unoccupied edge $y$ connecting this alive endpoint to another alive vertex, then Maker (or Breaker) also has a winning strategy by claiming the edge $y$ instead.
\end{lemma}

\begin{proof}
	Let $v$ be the common alive endpoint shared by edges $x$ and $y$, let $u$ be the dead endpoint of $x$, and let $w$ be the alive endpoint of $y$ ($w \neq v$). 
	From Breaker's view, claiming the edge $x = \{u,v\}$ can only contribute to neutralizing the vertex $v$. On the other hand, from Maker's side, claiming $x$ helps in accumulating edges at $v$, while $u$ provides no further value since it cannot be fully claimed anymore.
	
	Now suppose a player has a winning strategy $\mathcal{S}$ that dictates playing the edge $x$. We can construct a modified strategy $\mathcal{S}'$ where the player claims the edge $y = \{v,w\}$ instead of $x$. Since $w$ is alive, claiming $y$ affects two alive vertices ($v$ and $w$), whereas claiming $x$ only affects one alive vertex ($v$) and one already dead vertex ($u$). Thus, any strategic advantage gained at $v$ by claiming $x$ is identically preserved by claiming $y$, while the player gains an additional potential advantage at the vertex $w$. It follows that if $\mathcal{S}$ is a winning strategy, $\mathcal{S}'$ must also be a winning strategy.
\end{proof}

Figure~\ref{Q3pr} illustrates the case analysis for the proof of Lemma~\ref{cube}, where $M_i$ and $B_i$ denote the $i$-th move of Maker and Breaker, respectively. Due to the symmetries of $Q_3$, we can assume without loss of generality that Maker starts the game by claiming an arbitrary edge $M_1$, and Breaker responds by claiming an adjacent edge $B_1$. 

Following this initial round, Maker faces seven structurally non-isomorphic options for their second move, labeled from $I$ to $VII$ in the upper-left corner of the figure. Crucially, any vertex incident to Breaker's claimed edge $B_1$ becomes immediately dead. Therefore, according to Lemma~\ref{jobb}, we can significantly reduce the branching of the case analysis: Maker will always prefer to claim an edge connecting two alive vertices over an edge incident to only one of the previous two alive vertices and a dead vertex, whenever such an option is available.

In the first three cases (labeled $I$ to $III$), the analysis branches into three or four subcases based on Maker's third move. However, after at most three rounds for cases $I$--$III$, or just two rounds for cases $IV$--$VII$, Breaker is able to establish a winning pairing strategy on the remaining unclaimed edges. These strategies are explicitly indicated by arrows in Figure~\ref{Q3pr}. To aid readability, empty circles in the diagram denote alive vertices, while solid black circles represent vertices that are already dead at that stage of the game.
\end{proof}

As the core insight of this paper, to extend our results to higher-dimensional geometries (such as $Q_d, \mathbb{Z}^d, \mathbb{T}^d_n$, etc.), we require the following three lemmas. Here, $H \square G$ denotes the Cartesian product of graphs $H$ and $G$, while $H \cup G$ denotes their union.

\begin{lemma} \label{product}
	For any simple graphs $H$ and $G$, we have
	\[ \delta_B(H \square G) \ge \delta_B(H) + \delta_B(G). \]
\end{lemma}

\begin{lemma} \label{union}
	If $H$ and $G$ are edge-disjoint graphs on the same vertex set, then 
	\[ \delta_B(H \cup G) \ge \delta_B(H) + \delta_B(G). \]
\end{lemma}

\begin{lemma} \label{disj}
	Let $G = \bigsqcup_{i \in I} H_i$ be a graph formed by the disjoint union of a (finite or infinite) family of graphs $\{H_i\}_{i \in I}$. Then 
	$\delta_B(G) = \min_{i \in I} \delta_B(H_i). $
	In particular, if $G$ is a disjoint union of $3$-dimensional hypercubes, then $\delta_B(G) \ge 1$.
\end{lemma}

\begin{proof}[Proof of Lemma~\ref{product}]
	The statement follows from a simple reactive strategy across the orthogonal layers of the Cartesian product. By definition, the edge set of $H \square G$ can be partitioned into two disjoint sets: $E_H$ and $E_G$ which consists of the copies of the edges of $H$ and $G$, respectively. Thus, $E(H \square G) = E_H \cup E_G$ and $E_H \cap E_G = \emptyset$.
	
	Let $\mathcal{S}_H$ and $\mathcal{S}_G$ be Breaker's optimal winning strategies that guarantee at least $\delta_B(H)$ edges at every vertex of $H$ and at least $\delta_B(G)$ edges at every vertex of $G$, respectively. Breaker plays the game on $H \square G$ by treating the two layers independently.
	
	Since the two edge sets are completely disjoint, Maker cannot interfere with the degree accounting across different layers. At the end of the game, for any vertex $v = (v_H, v_G) \in V(H \square G)$, the total degree secured by Breaker is the sum of the degrees secured in the subgames. Consequently, 
	$\deg_{\text{Breaker}}(v) \ge \delta_B(H) + \delta_B(G), $
	which implies $\delta_B(H \square G) \ge \delta_B(H) + \delta_B(G)$.
\end{proof}

\begin{proof}[Proof of Lemma~\ref{union}]
	The proof is analogous to that of Lemma~\ref{product}. Since $H$ and $G$ are edge-disjoint graphs defined on the exact same vertex set $V$, we have $E(H \cup G) = E(H) \cup E(G)$ with $E(H) \cap E(G) = \emptyset$. 
	Breaker can decouple the game on $H \cup G$ into separate, non-interfering subgames played on $H$ and $G$. At the end of the game, every vertex $v \in V$ accumulates edges from both independent subgames.
	This yields $\delta_B(H \cup G) \ge \delta_B(H) + \delta_B(G)$, completing the proof.
\end{proof}

\begin{proof}[Proof of Lemma~\ref{disj}]
	Since the components $H_i$ are pairwise vertex-disjoint and edge-disjoint, the game played on $G$ completely decouples into independent subgames played on each component. Whenever Maker claims an edge in a specific component, Breaker responds within the same component, following the optimal strategy.  
\end{proof}

\begin{proof}[Proof of Theorem~\ref{d-cube}]
	Any dimension $d \ge 1$ can be represented as $d = 3k + r$, where $k = \lfloor d/3 \rfloor$ and $r \in \{0, 1, 2\}$. This allows us to factor the hypercube as the Cartesian product of $k$ copies of $3$-dimensional hypercubes and a remaining factor $Q_r$:
	$Q_d = \underbrace{Q_3 \square Q_3 \square \dots \square Q_3}_{k \text{ times}} \square \, Q_r.$
	By Lemma~\ref{cube}, $\delta_B(Q_3) \ge 1$. By applying Lemma~\ref{product} over all factors, we obtain:
	$\delta_B(Q_d) \ge \left( \sum_{i=1}^{k} \delta_B(Q_3) \right) + \delta_B(Q_r) = k = \left\lfloor \frac{d}{3} \right\rfloor. $
\end{proof}

\begin{proof}[Proof of Theorem~\ref{d-grid}]
	Let $d = 3k + r$ again. Let $G^d$ denote either the infinite grid $\mathbb{Z}^d$ or the toroidal grid $\mathbb{T}^d_n$ with $n \ge 2$. Utilizing the Cartesian product factorization, we can decompose the graph into $k$ copies of $3$-dimensional base components and a remaining $r$-dimensional factor: 
	$G^d = \underbrace{G^3 \square G^3 \square \dots \square G^3}_{k \text{ times}} \square \, G^r.$
	To evaluate the base case $G^3$ (which is $6$-regular), we partition its edge set based on the parity of its localized grid coordinates into two edge-disjoint subgraphs $G_1$ and $G_2$. Geometrically, each $G_i$ ($i \in \{1,2\}$) forms a disjoint union of $3$-dimensional hypercubes. By Lemma~\ref{disj}, Breaker guarantees at least one edge per vertex in each layer. Then Lemma~\ref{union} implies $\delta_B(G^3) \ge 2$.
	By Lemma~\ref{product}, Breaker decouples the global game into the $k$ independent $6$-regular $G^3$ layers and the remaining $2r$-regular $G^r$ layer. On the $G^3$ components, applying the base strategy guarantees a degree of $2k$. On $G^r$, Breaker deploys the classical Euler-pairing strategy, which secures $\lfloor 2r/4 \rfloor = \lfloor r/2 \rfloor$ edges at every vertex.
	 
	\[ \delta_B(G^d) \ge \left( \sum_{i=1}^{k} \delta_B(G^3) \right) + \delta_B(G^r) \ge 2k + \left\lfloor \frac{r}{2} \right\rfloor = \lfloor 2d/3 \rfloor. \]
	
	Thus, $\delta_B(G^d) \ge \lfloor 2d/3 \rfloor$ holds for both $\mathbb{Z}^d$ and $\mathbb{T}^d_n$, completing the proof.
\end{proof}

\begin{observation} \label{univ}
	By repeatedly applying Lemma~\ref{product}, \ref{union} and \ref{disj}, one can observe that the $\lfloor d/3 \rfloor$ lower bound is not restricted to hypercubes and grids. Indeed, by taking any collection of small-order $3$-regular graphs on which Breaker has a winning strategy (such as $Q_3$, $K_4$, $K_{3,3}$, or $W$) and iteratively forming their Cartesian products or edge-disjoint unions, we can generate an infinite family of graphs where Breaker universally secures at least $\lfloor d/3 \rfloor$ edges at every single vertex.
\end{observation}

\section{Further Remarks and Open Problems}

Any result that establishes $\delta_B(Q_d) > d/3$ for some specific dimension $d$ can be scaled up to higher dimensios. Consequently, Theorem~\ref{d-cube} could be systematically improved if stronger base cases were found. For instance, demonstrating that $\delta_B(Q_5) = 2$ would immediately imply that $\delta_B(Q_d) \ge \lfloor 2d/5 \rfloor$. Unfortunately, it can be formally observed that $\delta_B(Q_5) = 1$, as demonstrated by Maker's winning strategy illustrated in Figure~\ref{Q5}. The analysis branches into two main scenarios, depending on whether Breaker responds to Maker's opening move by claiming an adjacent edge or not.

Regarding the next potential dimension for improving the $d/3$ baseline, demonstrating that $\delta_B(Q_8) = 3$ would immediately imply $\delta_B(Q_d) \ge \lfloor 3d/8 \rfloor$. Unfortunately, to check the $8$-dimensional hypercube graph is hopeless even by programming.

\begin{figure}[htbp]
	\centering
	\includegraphics[scale=0.5]{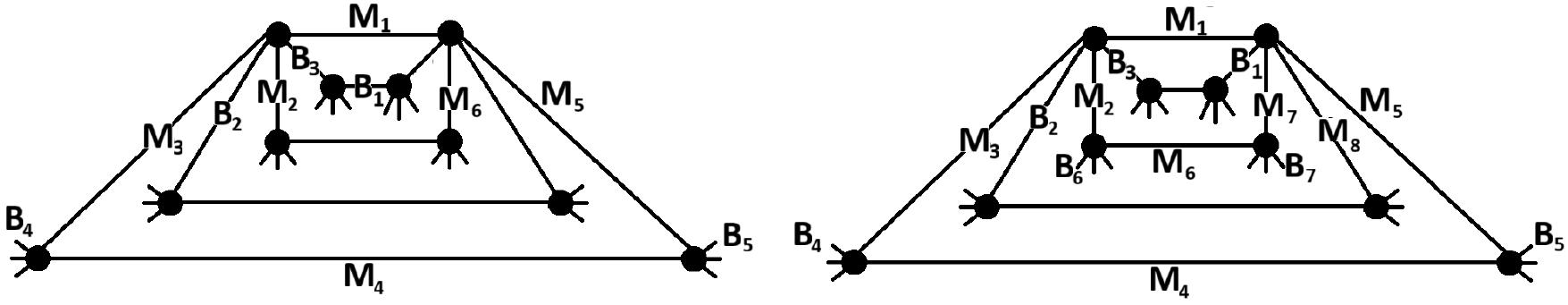}
	\caption {Two cases, Breaker cannot win the 2-Degree Game on $Q_5$.}
	\label{Q5}
\end{figure} 

Finally, several intriguing open questions arise if we modify the game rules to restrict Maker's movement. For instance, one could consider a variant where Maker is required to claim edges such that their induced subgraph remains connected at all times, mimicking the connectivity constraints in London and Pluh\'ar \cite{Lond}. Or Walker-Breaker games, where Maker must claim edges sequentially along a continuous path or walk in $G$ (see \cite{Esp, Forc1}). Chooser-Picker (or Client-Waiter) games represents a promising future research as well, as in case of other graph games \cite{CGH, Dvo, Dvo2, HKT1, HKT2, Kriv2}.

Another fundamental direction would be to characterize the class of all $3$-regular graphs on which Breaker has a winning strategy in the $1$-degree game.

\section*{Acknowledgments}
The author would like to express their deepest gratitude to Andr\'as Pluh\'ar. His insightful comments, valuable suggestions, and constant encouragement were indispensable throughout the brainstorming sessions and the writing of this paper.

Declaration on the use of generative AI. Gemini was used in a limited manner during manuscript preparation for language editing and organizational assistance. 




\begin{thebibliography}{99}

\bibitem{BP} J. Balogh, A. Pluh\'ar,
The positive minimum degree game on sparse graphs.
{\em The Electronic Journal of Combinatorics} {\bf 19} \#P22 (2012).

\bibitem{BMP} J. Balogh, R. Martin and A. Pluh\'ar,
The diameter game.
{\em Random Structures and Algorithms} Vol. {\bf 35}, 369--389 (2009).

\bibitem{Btree} J. Balogh, Private communication.

\bibitem{Beck2} J. Beck,
Deterministic graph games and a probabilistic intuition.
{\em Combinatorics, Probability and Computing} {\bf 3}, 13--26 (1994).

\bibitem{BeckLLL} J. Beck, 
An algorithmic approach to the Lovász local lemma. 
{\em I. Random Structures \& Algorithms,} {\bf 2(4)}, 343--365 (1991).

\bibitem{BB} J. Beck, {\em Combinatorial games. Tic-Tac-Toe Theory.} Cambridge University Press, 2008.

\bibitem{CGH} D. Clemens, P. Gupta, F. Hamann, A. Haupt, M. Mikalacki and Y. Mogge, Fast strategies in Waiter-Client games on $K_n$. {\em The Electronic Journal of Combinatorics,} {\bf 27(3)}, 1–-35 (2020).

\bibitem{Cser} 
A. Csernenszky, The Picker–Chooser diameter game. 
{\em Theoretical Computer Science,} {\bf 411(40-42)}, 3757--3762 (2010).

\bibitem{TI} C. Dowden, M. Kang, M. Mikalački, and M. Stojaković. The Toucher-Isolator game. {\em The Electronic Journal of Combinatorics,} 26(4):P4.6. (2019) https://doi.org/10.37236/8617

\bibitem{Dvo} V. Dvorák, 
Waiter–Client triangle-factor game on the edges of the complete
graph. {\em European Journal of Combinatorics,} 96:103356 (2021).

\bibitem{Dvo2} V. Dvorák,  
Waiter-Client clique-factor game. {\em Discrete Mathematics,} 346(1), 113191 (2023).

\bibitem{erdos} P. Erd\H{o}s and J. L. Selfridge,
On a combinatorial game.
\textit{Journal of Combinatorial Theory Series A} \textbf{14}, 298--301 (1973).

\bibitem{Esp} L. Espig, A. Frieze, M. Krivelevich, W. Pegden,  Walker-breaker games. {\em SIAM Journal on Discrete Mathematics,} {\bf29(3)}, 1476--1485, (2015).

\bibitem{Forc1} J. Forcan, M. Mikalački, Spanning structures in Walker–Breaker games. {\em Fundamenta Informaticae,} {\bf 185}(1), 83--97 (2022).

\bibitem{Forc2} J. Forcan, M. Mikalački.  Maker-Breaker total domination game on cubic graphs. {\em Discrete Mathematics \& Theoretical Computer Science,} 24(Graph Theory), (2022).

\bibitem{GSZ} 
H. Gebauer and T. Szab\'o,  
Asymptotic random graph intuition for the biased connectivity game. 
{\em Random Structures \& Algorithms,} {\bf 35(4)}, 431--443 (2009).

\bibitem{Geb} 
H. Gebauer, 
Disproof of the neighborhood conjecture with implications to SAT. 
{\em Combinatorica,} {\bf 32(5)}, 573--587 (2012).


\bibitem{Hefetz} D. Hefetz, M. Krivelevich, M. Stojakovi\'c and
T. Szab\'o, Global Maker-Breaker games on sparse graphs.
{\em European Journal of Combinatorics} {\bf 32}, 162--177 (2011).

\bibitem{HKT1} D. Hefetz, M. Krivelevich and W. E. Tan, 
Waiter–Client and Client–Waiter planarity, colorability and minor games. {\em Discrete Mathematics,} {\bf 339(5)}, 1525--1536 (2016).

\bibitem{HKT2} D. Hefetz, M. Krivelevich and W. E. Tan, 
Waiter–Client and Client–Waiter Hamiltonicity games on random graphs. {\em European Journal of Combinatorics,} {\bf 63,} 26--43 (2017).

\bibitem{Knox} F. Knox, Two constructions relating to conjectures of Beck on positional games. {\em arXiv preprint arXiv:} 1212.3345 (2012). 


\bibitem{KSZ} 
M. Krivelevich, M. and T. Szab\'o,  
Biased positional games and small hypergraphs with large covers. 
{\em The Electronic Journal of Combinatorics,} R70-R70 (2008).


\bibitem{Kriv2} M. Krivelevich and O. Dean. Client–Waiter Games on Complete and Random Graphs. {\em The Electronic Journal of Combinatorics,} 23(4):P4.38 (2016).

\bibitem{Lond} A. London and A. Pluhár, Spanning Tree Game as Prim Would Have Played. {\em Acta Cybernetica} {\bf 23}(3) 921--927 (2018).

\bibitem{MT} R. A. Moser and G. Tardos, 
A constructive proof of the general Lovász local lemma. 
{\em Journal of the ACM (JACM),} {\bf 57(2)}, 1--15 (2010).


\bibitem{NS} R. Naimi and E. Sundberg. Pairing strategies for the Maker–Breaker game on the hypercube with subcubes as winning sets. {\em Discrete Mathematics} 346(10), 113576 (2023).

\bibitem{Szekely} L. A. Sz\'ekely,
On two concepts of discrepancy in a class of combinatorial
games. Finite and Infinite Sets, Colloq. Math. Soc. J\'anos Bolyai,
Vol. 37, North-Holland, 679--683 (1984).


\end{thebibliography}
\end{document}